\documentclass[12pt]{article}
\usepackage[ansinew]{inputenc}

\usepackage{latexsym,amsmath,amsthm,amssymb,amsfonts,amscd}
\usepackage{pstricks,pst-node,pst-plot}

\usepackage{graphicx}
\usepackage{color}

\usepackage{graphics}

\date{\today}

\newcommand{\indic}[1]{\mathbf{1}_{#1}}
\newcommand{\indica}[1]{\mathbf{1}_{\{#1\}}}

\newtheorem{theorem}{Theorem}
\newtheorem{conj}{Conjecture}

\newtheorem{coro}[theorem]{Corollary}
\newtheorem{proposition}[theorem]{Proposition}
\theoremstyle{definition}

\newtheorem{remark}[conj]{Remark}

\newcommand{\eps}{\varepsilon}
\newcommand{\E}{\mathbb{E}}
\renewcommand{\P}{\mathbb{P}}
\newcommand{\R}{\mathbb{R}}

\newcommand{\bck}{\!\!\!}
\newcommand{\bx}{{\bf x}}
\newcommand{\FF}{\mathcal{F}}
\newcommand{\La}{\Lambda}
\newcommand{\X}{\Xi}

\newcommand{\iy}{\infty}

\newcommand{\N} {\mathbb{N}}

\newcommand{\sumd}{\sum_{i=1}^\iy x_i^2}

\begin{document}

\author{Vlada Limic\thanks{Supported in part by ANR MANEGE grant.}\\
CNRS UMR 6632, Universit\'e de Provence
}

\title{Genealogies of regular exchangeable coalescents with applications to sampling}

\maketitle

\begin{abstract}
\noindent
This article considers a model of genealogy corresponding to a regular exchangeable coalescent (also known as $\X$-coalescent) started from a large finite configuration, and undergoing neutral mutations. 
Asymptotic expressions for the number of active lineages were obtained by the author in a previous work.
Analogous results for the number of active mutation-free lineages and the
combined lineage lengths are derived using the same martingale-based technique.
They are given in terms of convergence in probability, while extensions to convergence in moments and convergence 
almost surely are discussed.
The above mentioned results have direct consequences on the sampling theory in the $\X$-coalescent setting. 
In particular, the regular $\X$-coalescents that come down from infinity (i.e., with locally finite genealogies)
have an asymptotically equal number of families under the corresponding infinite alleles and infinite sites models.
In special cases, quantitative asymptotic formulae 
for the number of families that contain a fixed number of individuals can be given. 
\end{abstract}



\noindent {\em Running Head:}
Genealogies of regular exchangeable coalescents

\clearpage

\section{Introduction}
Kingman's coalescent \cite{king82,king82b} is one of the 
central models of mathematical population genetics.
From the theoretical perspective, its importance is 
linked to the duality with 
the Fisher-Wright diffusion (and more generally with the Fleming-Viot process).
Therefore the Kingman coalescent emerges in the scaling limit of genealogies
of all evolutionary models that are asymptotically linked to Fisher-Wright
diffusions.
From the practical perspective, its elementary nature allows for exact
computations and fast simulation, making it amenable to statistical analysis.

Assume that the original sample has $n$ individuals, labeled $\{1,2,\ldots,n\}$.
The {\em genealogy} is a process in which ancestral lineages 
coalesce in continuous time.
One can identify the original sample with the trivial partition $\{\{1\},\ldots,\{n\}\}$,
and moreover, at any positive time, one can identify each of the active ancestral
lineages with a unique equivalence class of $\{1,2,\ldots,n\}$ that consists of the labels of all the individuals 
that descend from this lineage.
In this way, each coalescent event of two (or more) ancestral lineages can be perceived as 
the merging of the corresponding equivalence classes.
Ignoring the partition structure information, 
one can now view the coalescent as a {\em block} (rather than equivalence class) merging
process.

Kingman's coalescent corresponds to the dynamics where
each pair of blocks coalesces at rate 1.
In particular, while there are $n$ blocks present in the configuration, 
the total number of blocks decreases by $1$ at rate ${n\choose 2}$.
In 1972 Ewens \cite{ewens_sam} derived a sampling formula that holds for several neutral population evolution models, and in particular for the Kingman coalescent. Its importance is well indicated by almost 900
scientific citations in less than four decades\footnote{source: ISI
  Web of Knowledge {\tt http://apps.isiknowledge.com/}}.
The Ewens sampling formula will be recalled 
in Section \ref{S:discuss}. 

The fact that in the Kingman coalescent dynamics only pairs of blocks
can merge at any given time
makes it less suitable to 
model evolution of marine populations or
viral populations under strong selection.
In fact, it is believed (and argued to be observed in experiments, see e.g.~\cite{marine}) that 
in such settings the reproduction mechanism allows for a proportion of 
the population to have the same parent (i.e., first generation ancestor).
This translates to having multiple
collisions of
the ancestral lineages in the corresponding coalescent mechanism.

A family of mathematical models with the above property was independently introduced and
studied by Pitman \cite{pit99} and Sagitov \cite{sag99} (see also the remark on page 195 of Donnelly and Kurtz \cite{donkur99})
under the name {\em $\Lambda$-coalescents} or
{\em coalescents with multiple collisions}.
Almost immediately emerged 
an even more general class of models, named {\em $\Xi$-coalescents} or 
{\em coalescents with simultaneous multiple collisions} or {\em exchangeable coalescents}.
The $\X$-coalescent processes were 
initially studied by  M\"{o}hle and Sagitov \cite{moehle sagitov}, and
introduced by Schweinsberg \cite{schweinsberg_xi} in their full generality.
In particular, it is shown in \cite{moehle sagitov} that any limit
of genealogies arising from a population genetics model with (time-homogeneous)  exchangeable reproduction
mechanism must be a $\Xi$-coalescent.
For additional pointers to the recent coalescent literature see \cite{berest brasil,bertoin book}.

The present article can be considered as a sequel to Berestycki et al.~\cite{bbl1} and Limic \cite{me_xi}.
It demonstrates once again the power of martingale techniques 
in the study of exchangeable coalescents.
In fact, the main result is proved by a variation of the technique from
\cite{bbl1,me_xi}.
The primary interest of the current work however is 
its potential for applications in studies of sampling statistics.

Unlike \cite{bbl1,me_xi,bbl2}, the present work does not focus primarily on 
the coalescents that ``come down from infinity''. 
Indeed, the starting configuration is finite in the current context, so as long
as a certain regularity condition (see \cite{me_xi} or (\ref{ERcond}) at the beginning of Section \ref{S:results}) holds, 
the small-time asymptotic results derived here (in Theorem \ref{TMinit} and Proposition \ref{PN}) apply.
All the $\La$-coalescents (as well as certain general exchangeable coalescents of 
particular interest in mathematical 
population genetics \cite{dursch,schdur}) are regular in this sense
(see also Remark 12 in \cite{me_xi}).
If a regular $\X$-coalescent has a locally finite genealogy (or equivalently, if its standard version comes
down from infinity),
its small-time asymptotics determines the asymptotic growth of the corresponding combined lineage length,
and in turn, the asymptotic 
number of families in the infinite alleles (resp.~sites) model (see Theorem \ref{Tfull limit}).
It is worth pointing out that Proposition \ref{PN} is the first example of a
``meta-theorem'' that applies in the more general context of regular exchangeable coalescents that have (potentially random but) uniformly bounded rate (per unit lineage length) of mutation, recombination and/or migration (see also Remark \ref{Rrobust}).
Some consequences on the asymptotic frequency spectrum are discussed in Section \ref{S:discuss}.

This paper inherits the basic notation from \cite{me_xi}.
For the benefit of the reader we recall it here.
The set of real numbers is denoted by $\R$ and $(0,\iy)$ by $\R_+$.
For $a,b\in \R$, denote by $a \wedge b$ (resp.~$a \vee b$) the minimum (resp.~maximum)
of the two numbers.
Let
\begin{equation}
\label{EDel}
\Delta := \{(x_1,x_2,\ldots,): x_1\geq x_2 \geq \ldots \geq 0,\, \sum_i x_i \leq 1\},
\end{equation}
be the infinite unit simplex.
For $\bx=(x_1,x_2,\ldots) \in \Delta$ and $c\in \R$, let 
\[
c\, \bx = (c x_1,c x_2,\ldots).
\]
The notation $\log$ is reserved for the {\em natural} logarithm, that is, the inverse of $\R \ni x \mapsto e^x \in (0,\iy)$. 
If $f$ is a function, defined in a left-neighborhood $(s-\eps,s)$ of a point $s$,
 denote by $f(s-)$ the left limit of $f$ at $s$.
Given two functions $f,g:\R_+\to \R_+$, write
$f=O(g)$ if $\limsup f(x)/g(x) <\infty$,
$f=o(g)$ if $\limsup f(x)/g(x) =0$,
 and $f\sim g$ if $\lim f(x)/g(x) =1$.
Furthermore, write $f=\Theta(g)$ if both $f=O(g)$ and $g=O(f)$.
The point at which the limits are taken
is determined from the context.
If $\FF=(\FF_t,t\ge 0)$ is a filtration, and $T$ a stopping time
relative to $\FF$, denote by $\FF_T$ the standard filtration generated by $T$,
see for example \cite{durrett}, page 389.

The rest of the paper is organized as follows: 
the model 
of interest is described in Section \ref{S:prelim}, 
in Section \ref{S:results} the main results are stated, followed by a discussion of some consequences,
while the arguments are postponed until Section \ref{S:martingale}.

\section{Definitions and preliminaries}
\label{S:prelim}

\subsection{The model}
\label{S:model}
For the purposes of this note, the precise definition  of 
exchangeable coalescent processes (and their standard infinite version) 
is not essential. An interested reader can consult any of the
references \cite{schweinsberg_xi,berest brasil,bertoin book,me_xi} for details of the
construction, and further properties.

Instead, we next construct a related ``genealogical'' process which, together with a 
suitable enrichment, will be in the focus of the present study.
Suppose that $\X$ is a probability measure on $\Delta$.
Assume that we are given 
 a Poisson Point Process on $\R_+ \times \Delta$
\begin{equation}
\label{DPPPpi}
\pi(\cdot) = \sum_{k \in \N } \delta_{t_k,\bx_k}(\cdot),
\end{equation}
with intensity measure $dt\, \otimes\, \X'(d\bx)/\sumd$, 
where 
$\X(d\bx)=(1-\X'(\Delta)) \delta_{(0,0,\ldots)}(d\bx)$ $+ \X'(d\bx)$
and $\X'((0,0,\ldots))=0$, $\X'(\Delta)\in (0,1]$ (the trivial case $\X'(\Delta)=0$ corresponds to the Kingman coalescent).

Let $n$ be a finite (typically large) integer.
The {\em $n$-genealogy} (associated to $\Xi$) evolves as follows:
at the initial time $0$, there are $n$ branches present in the system, each having trivial length $0$. 
As time increases, the length of each branch increases at unit rate.
An atom  $(t,\bx)$ of $\pi$ influences the evolution as follows:
at time $t$ let each
branch present in the system at time $t-$ choose an i.i.d.~color from $\N \cup (0,1)$, independently of the past 
evolution, according to the common law
\[
\label{Dcoloring}
P_\bx(\{i\})=x_i, \ i\geq 1 \ \mbox{ and } P_\bx(du)=(1-\sum_{i=1}^\iy x_i)\,du,\ u\in (0,1).
\]
For each $j\geq 1$, all the branches of color $j$ collapse immediately (at time $t$)
 into a single branch. 
In addition, each pair of branches coalesces at rate $1-\X'(\Delta)$ independently of the above color and collapse procedure.
The readers might notice that whenever $\int_\Delta \X'(d\bx)/\sumd=\infty$, 
the time-projection of the set of atoms of $\pi$ is dense on any interval 
of positive length. 
This could a priori cause difficulties in the construction,
however, the above branch growing and collapsing process is always well-defined, 
as can be directly verified (or check \cite{me_xi,schweinsberg_xi} for a similar construction of exchangeable coalescents).
Furthermore, it is important to note that the just described construction can be coupled over $n \in \N$
in such a way that 
for any two $m,n$, where $m>n$, restricting the process started from $m$ distinct branches onto the first 
$n$ branches
gives precisely the process started from $n$ branches.
We call this coupling the {\em full genealogy} or the {\em full $\Xi$-genealogy}.

\begin{remark}
\label{R:scale}
Setting $\X(\Delta)=1$ is convenient but arbitrary.
It is easily seen that the time scale of any exchangeable coalescent can be multiplied 
by a constant factor, 
so that the law of the resulting process matches that of an exchangeable coalescent corresponding
to a driving measure 
$\X$ of total mass $1$.
\hfill $\Diamond$
\end{remark}

One can enrich the above construction as follows: let $\upsilon$ be a completely 
independent Poisson Point Process of marks that
arrive at rate $\gamma$ per unit length. 
This means that while there are $m$ branches present in the
system, the next mark (mutation) arrives at rate $m \gamma$, and when it arrives it is placed uniformly 
at random (and independently from everything else) onto one of the branches.
The parameter $\gamma$ is usually called the {\em mutation rate}. 
In this enriched construction,
each branch is in one of the two states, {\em open} or {\em closed}, at any given time. 
Initially all the branches are open.
A branch becomes closed starting from the moment a mark arrives to it, additional marks that 
possibly arrive afterward onto it do not change its state.
Immediately after two or more branches collapse into one, the state of this newly formed branch is
determined as follows: if at least one of the contributing branches is open, the new branch is open, otherwise
it is closed.
It is clear that one can couple this enriched genealogy
again over $n\in \N$, and we call the resulting process the {\em full
  marked genealogy}. 

Let $\FF=(\FF_t,\,t\geq 0)$ be the filtration generated by the full
marked genealogy. 
Denote by $N\equiv N^n$ the branch counting process, and by $N^o\equiv N^{n;o}$ (resp.~$N^c\equiv N^{n;c}$)
the open (resp.~closed) branch counting process.
Note that, for each $n$, all the processes $N^n,N^{n;o}, N^{n;c}$ are $\FF$-adapted.
It is clear that
\[
N(t)= N^o(t) + N^c(t), \ t\geq 0.
\]
Moreover, note that both $N$ and $N^o$, unlike $N^c$, are monotone (decreasing) processes.
In addition we have that, almost surely, for each $t\geq 0$ and $n\geq 1$, 
\[ 
N^n(t) \leq N^{n+1}(t)
\ \mbox{ and } \
N^{n;o}(t) \leq N^{n+1;o}(t),
\]
but it is not necessarily true that $N^{n;c}(t) \leq N^{n+1;c}(t)$ (on Figure 1, $N^{1;c}(t_1)>N^{2;c}(t_1)$). 
If for each $t>0$ the family $(N^n(t),n\geq 1)$ is tight, or equivalently if 
$\lim_n N^n(t)$ exists and is a finite random variable, we will say that
the full $\Xi$-genealogy is {\em locally finite}.
This property is equivalent to the above mentioned coming down from infinity property 
for the (standard) $\Xi$-coalescent.

   \psset{linewidth=0.4pt}
   \psset{xunit= 1.05cm,yunit= 0.7cm}
   \pspicture(-3.5,-1.5)(6,9)

\psline(0,0)(1.5,0)
\psline(0,1)(1.5,1)
\psline(0,2)(1.5,2)
\psline[linewidth=0.8pt,linestyle=dashed,dash=2pt 2pt](1.5,0)(1.5,2)
\psline(0,3)(2.5,3)
\psline(0,4)(2.5,4)
\psline[linewidth=0.8pt,linestyle=dashed,dash=2pt 2pt](2.5,3)(2.5,4)
\psline(0,5)(1.5,5)
\psline(0,6)(1.5,6)
\psline(0,7)(1.5,7)
\psline(0,8)(1.5,8)
\psline[linewidth=0.8pt,linestyle=dashed,dash=2pt 2pt](1.5,5)(1.5,8)
\psline(1.5,1)(4,1)
\psline(2.5,3.5)(4,3.5)
\psline(1.5,6.5)(4,6.5)
\psline[linewidth=0.8pt,linestyle=dashed,dash=2pt 2pt](4,1)(4,6.5)
\psline(4,4)(6,4)
\pscircle*(0.6,0){0.09}
\pscircle*(0.4,2){0.09}
\pscircle(1.2,2){0.09}
\pscircle*(2.1,3){0.09}
\pscircle*(1,6){0.09}
\pscircle*(0.3,7){0.09}
\pscircle*(2.3,6.5){0.09}
\pscircle(3.2,6.5){0.09}
\pscircle(3.7,6.5){0.09}
\pscircle*(5,4){0.09}
\psset{linewidth=1.5pt}
\psline(0,0)(0.6,0)
\psline(0,1)(1.5,1)
\psline(0,2)(0.4,2)
\psline(0,3)(2.1,3)
\psline(0,4)(2.5,4)
\psline(0,5)(1.5,5)
\psline(0,6)(1,6)
\psline(0,7)(0.3,7)
\psline(0,8)(1.5,8)
\psline(1.5,1)(4,1)
\psline(2.5,3.5)(4,3.5)
\psline(1.5,6.5)(2.3,6.5)
\psline(4,4)(5,4)
\rput[r](-0.3,0){\small $1$}
\rput[r](-0.3,1){\small $2$}
\rput[r](-0.3,2){\small $3$}
\rput[r](-0.3,3){\small $4$}
\rput[r](-0.3,4){\small $5$}
\rput[r](-0.3,5){\small $6$}
\rput[r](-0.3,6){\small $7$}
\rput[r](-0.3,7){\small $8$}
\rput[r](-0.3,8){\small $9$}
\rput[t](6,-0.7){\small time}
\rput[t](0,-0.7){\small $0$}
\rput[t](1.5,-0.7){\small $t_1$}
\rput[t](2.5,-0.7){\small $t_2$}
\rput[t](4,-0.7){\small $t_3$}
\endpspicture

\noindent
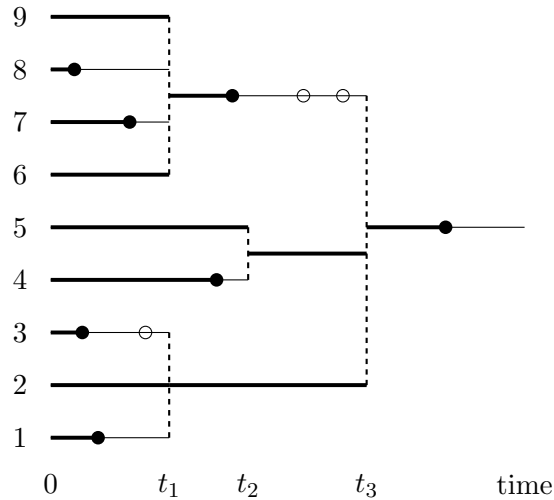
{\small Figure 1: A realization with $n=9$ particles labeled by $\{1,\ldots,9\}$, and coalescent events 
occurring at times 
$t_1, t_2$ and $t_3$. Open branches are indicated as thicker lines,
and any mutation that falls onto an open (resp.~closed) line is depicted as filled 
(resp.~unfilled) circle.}
\setcounter{figure}{1}

\vspace{0.2cm}
Call a mutation (or mark) that arrives onto an open (resp.~closed) branch {\em open} (resp.~{\em closed}).
Furthermore, denote by $M\equiv M^n$ the mutation counting process,
 and by $M^o\equiv M^{n;o}$ (resp.~$M^c\equiv M^{n;c}$)
the open (resp.~closed) mutation counting process.

Clearly,
\[
M(t)= M^o(t) + M^c(t), \ t\geq 0.
\]
For the realization in Figure 1, $N(t_1) =4=N^o(t_1)$, while $M(t_1)=5$ and $M^o(t_1)=4$.
We also have $N^c(t_1-)=4$, while $N^c(t_1)=0$ and $N^c(t_2-)=2$.

Let
\begin{equation}
\label{Etau one n}
\tau^n\equiv \tau_1^n := \inf \{ t \ge 0 : N^{n} (t) =1\},
\end{equation}
be the time of collapse to a single lineage,
and 
\[
\tau_*^n := \inf \{ t \ge \tau^n : \Delta \upsilon(t) >0 \},
\]
be the arrival time of the first mutation to this unique lineage.

In \cite{bbl2} it was noted that $M(\tau_*^n) =M(\tau^n)+1$  can be interpreted as the 
total number of ``families'' in the corresponding infinite sites model. For the realization
on Figure 1 we have $\tau^n=t_3$, and the ``families''
are $\{1\},\{3\}, \{3\}^*, \{4\}, \{7\}, \{8\},\{6,7,$ $8,9\},$ $\{6,7,8,9\}^*,$
$\{6,7,8,9\}^{**}$, $\{1,2,3,4,5,6,7,8,9\}$, where the superscripts indicate the difference of the two families
even though their contents, as subsets of $\N$, are identical.

Note in addition that $M^o(\tau_*^n)$ can similarly be interpreted as the total number 
of families in the infinite alleles model. 
For the realization on Figure 1, the (infinite alleles) families 
are $\{1\},\{2,5\}, \{3\},\{4\}, \{6,9\}, \{7\}, \{8\}$.

\section{Main results}
\label{S:results}
Suppose that the driving measure $\X$ satisfies
\begin{equation}
\label{ERcond}
\tag{R}
\int_\Delta \frac{(\sum_{i=1}^\iy x_i)^2}{\sumd} \,\X(d\bx) < \iy.
\end{equation}
This regularity condition already appeared in \cite{me_xi}. As observed in the introduction, all $\La$-coalescents 
(the Kingman coalescent included) are regular. 
\begin{remark}
In view of these observations, the ratio $(\sum_i x_i)^2/\sumd$ should be interpreted as $0$ at $(0,0,\ldots)$, so that the regularity of any $\X$-coalescent is determined by its ``non-Kingman'' part $\X'(d\bx)=\X(d\bx)\indic{\Delta \setminus \{(0,0,\ldots)\}}$. 

Presently, (\ref{ERcond}) seems necessary for the martingale analysis 
below to work, although certain modifications of the technique might lead to stronger results, see Remark 22 in \cite{me_xi}. 
We refer the reader to Section 3.2 in \cite{me_xi} for examples of non-regular coalescents.
\hfill $\Diamond$
\end{remark}
In analogy to (\ref{Etau one n}) define
\begin{equation}
\label{Etau b n}
\tau_b^n := \inf \{ t \ge 0 : N^{n} (t) \leq b\}, \  \tau_b^{n;o} := \inf \{ t \ge 0 : N^{n;o} (t) \leq b\}.  
\end{equation}
It is easy to see that (cf.~\cite{bbl1,me_xi})
$\tau_b^n \nearrow \tau_b$, almost surely, where 
\begin{equation}
\label{Etau fin}
\P(\tau_b \in (0,\iy])=1
\end{equation}
(note that 
$\tau_b$ takes value $\iy$ precisely in the cases where the genealogy is not locally finite).
Similarly, one can show via the
the full (marked) genealogy coupling that $\tau_b^{n;o}  \nearrow \tau_b^o$, where
\begin{equation}
\label{Etau o fin}
\P(\tau_b^o \in (0,\iy])=1.
\end{equation}
The fact that $\tau_b^{n;o}$ is non-decreasing in $n$ is a direct consequence of the coupling.
To see that the limit cannot take value $0$, note that
$\{N^{n;o}(s) > b\}=\{\tau_b^{n;o} > s\}$, and that
for each $s\geq 0$, $b\in \N$,
\[
\P(N^{n;o}(s) > b)\geq \P(N^n(s) >b, \mbox{the first $b+1$ branches stay open during }[0,s]),
\]
where the right hand side is non-decreasing in $n$, and where 
its limit $p(s,b)$ (monotone decreasing in both variables) satisfies
$\lim_{s\to 0} p(s,b)=1$, $b\geq 1$.

Let $\alpha \in (0,1/2)$ be arbitrary but fixed.
The main results of this article are given next, starting with the most general ones.
\begin{theorem}
\label{TMinit}
Under (\ref{ERcond}), there exists $n_0\in \N$ such that for each 
$t>0$ and $\beta \in (0,\alpha\wedge 1-2\alpha)$
\begin{eqnarray}
\label{EMinit one}
\sup_{n\in \N} E\left(\frac{M^{n;c}(t\wedge \tau_{n_0}^{n;o})}{\gamma\int_0^{t\wedge \tau_{n_0}^{n;o}} N^n(u)\,du}\right) = O( t^\alpha + t^{1-2\alpha}),\\ 
\label{EMinit two}
\limsup_n \P\left(\frac{M^{n;c}(t \wedge \tau_{n_0}^{n;o})}{M^n(t \wedge \tau_{n_0}^{n;o})} \geq  t^\beta \right) =O( t^{\alpha-\beta} \wedge t^{1-2\alpha-\beta}). 
\end{eqnarray}
\end{theorem}
\begin{remark}
Since the estimates are non-trivial only for $t\leq 1$, the optimal choice of $\alpha$ is $1/3$.
\hfill $\Diamond$
\end{remark}
This theorem is a consequence of the next result.
\begin{proposition}
\label{PN}
Under (\ref{ERcond}), there exists $n_0\in \N$ such that for each $s>0$
\begin{equation}
\label{Ebehav N}
P\left(\sup_{t\in[0,s]} \left| 
1-\frac{N^{n;o}(t\wedge \tau_{n_0}^{n;o})}{N^n(t \wedge \tau_{n_0}^{n;o})} \right|
> 8s^{\alpha}
\right) = O(s^{1-2\alpha}), \mbox{ uniformly over } n\in \N.
\end{equation}
\end{proposition}
\begin{remark}
Note that it is not true in general that $\lim_{n\to \iy} M^{n;o}(t)/M^{n}(t) = 1$ for a fixed 
$t>0$. For example, in the case where
$\Xi$ is a Dirac measure  $\delta_{(x,0,\ldots)}$ for some $x\in (0,1)$, all the original $n$ branches remain in the 
genealogy for an exponential (rate $1$) amount of time (regardless of $n$), therefore $N^n(t) \sim n$ and 
$N^{n;o}(t) \sim n e^{-\gamma t}$ at $t\approx 0$ (uniformly in $n$), implying
$M^n(t) \sim nt $ and 
$M^{n;o}(t) \sim (1- e^{-\gamma t})/\gamma n$.
In this case, $\lim_{n\to \iy} M^{n;o}(t)/M^{n,t} = f(t)$ where $\lim_{t\to 0} f(t)=1$.

All the $\X$-coalescents without proper frequencies (cf.~M\"ohle \cite{moehle_wofreq})
are regular, but they do not have locally finite genealogy.
In \cite{moehle_wofreq} it is shown that the appropriately scaled total lineage length of coalescents
without proper frequencies
converges in distribution to a non-trivial random variable.
\hfill $\Diamond$
\end{remark}
However, if the underlying full genealogy is locally finite,
one has $M^{n;o}(t)\sim M^{n}(t)$ for a fixed time $t$.
This result is stated next in a slightly different form, 
that may be more interesting from the perspective of applications.
Define
\begin{equation}
\label{D:psi xi}
\psi(q)\equiv \psi_{\X}(q):= \int_{\Delta}\frac{\sum_{i=1}^\iy(e^{-qx_i}-1+qx_i)}{\sumd}\,\X(d\bx).
\end{equation}
Note that the above integral converges since 
$e^{-z}-1+z\leq z^2/2$, $z\geq 0$,
and, in particular, 
\begin{equation}
\label{Epsi ineq}
\psi_\X(q)\leq q^2/2.
\end{equation}
Define $t \mapsto v^n(t)\in \R_+$ by
\begin{equation}
\label{Evn}
\int_{v^n(t)}^n \frac{dq}{\psi(q)}=t,
\end{equation}
and let
\[
\ell(n):=\int_1^n \frac{q}{\psi(q)}\,dq.
\]
\begin{theorem}
\label{Tfull limit}
If the full genealogy is locally finite and (\ref{ERcond}) holds, then
\begin{equation}
\label{Efull limit}
\lim_{n\to \infty} \frac{M^n(\tau^n)}{M^{n;o}(\tau^n)} 
= \lim_{n\to \infty} \frac{M^n(\tau^n)}{\gamma \cdot \ell(n)} 
= \lim_{n\to \infty} \frac{M^{n;o}(\tau^n)}{\gamma \cdot \ell(n)} 
= 1, \ \mbox{in probability.}
\end{equation}
\end{theorem}
\begin{remark}
In the $\La$-coalescent setting (where (\ref{ERcond}) automatically holds) this result 
coincides with the initial part of Theorem 2 in \cite{bbl2}.
The proof of this and related results in \cite{bbl2} is based on studying the 
asymptotic behavior of the arrival time of a
(uniformly chosen at) random mutation, as $n\to \iy$.
Under additional assumptions on $\La$, this asymptotics can be quite precisely determined, leading to the convergence almost surely results in \cite{bbl2} (also discussed in 
the paragraph containing (\ref{Eas verbal}) in Section \ref{S:discuss}).
The martingale-based technique presented here is not suitable for deducing precise information about the randomly chosen mutation.
However, it is more compact than the technique from \cite{bbl2} (that still relies on martingale estimates from \cite{bbl1}),
provides partial information even for non-locally finite genealogies, and is better suited for finding error bounds as well as extensions 
(like Theorem \ref{Tpartial limit} or the case of random uniformly bounded mutation rates as discussed in Remark \ref{Rrobust}).
\hfill $\Diamond$
\end{remark}

Define
\[
\ell_t(n) := \int_0^t v^n(u)\,du = \int_{v^n(t)}^n \frac{q}{\psi(q)} \,dq.
\]
Similar arguments as for Theorem \ref{Tfull limit} yield the following generalization:
\begin{theorem}
\label{Tpartial limit}
If the full genealogy is locally finite and (\ref{ERcond}) holds, then
for any bounded sequence $(t_n)_{n\geq 1}$  of positive numbers such that $\ell_{t_n}(n)$ diverges
as $n\to \iy$ 
\[
\lim_{n\to \infty} \frac{M^n(t_n)}{M^{n;o}(t_n)} 
= \lim_{n\to \infty} \frac{M^n(t_n)}{\gamma \cdot \ell_{t_n}(n)} 
= \lim_{n\to \infty} \frac{M^{n;o}(t_n)}{\gamma \cdot \ell_{t_n}(n)} 
= 1, \ \mbox{in probability.}
\]
\end{theorem}

\subsection{Consequences, extensions, and further comments}
\label{S:discuss}
Theorem \ref{Tfull limit} can be restated as follows: 
for any regular $\X$-coalescent that comes down from infinity, 
the number of families in the infinite alleles and the infinite sites model
are asymptotically equal, in probability.

If more is known about $\X$, the convergence of Theorem \ref{Tfull limit} can be extended to 
convergence almost surely, using arguments very similar to those of Section 4.4.1 in \cite{bbl1}. 
Furthermore, one could prove 
that
\[
\lim_{t\to 0}\frac{N^n(t)}{v^n(t)} = 1 \mbox{ in } L^p,\ \forall p \geq 1.
\]
Theorems \ref{Tfull limit}--\ref{Tpartial limit} could then be extended to convergence in 
the mean (see Remark \ref{R:ending} at the end of the article for details).

The asymptotics (\ref{Efull limit}) holds 
(see Drmota et al.~\cite{drmotaetal} and Basdevant and Goldschmidt \cite{basgol}) also for the
Bolthausen-Sznitman coalescent, which does not have locally finite genealogy.
The author is convinced that the arguments of Theorem \ref{Tfull limit} could be extended to cover
this special setting (where $\La(dx) =\X(d(x,0,0,\ldots))= dx$, $x\in [0,1]$), and provide a shorter 
(and more generic) proof of this result (see Remark \ref{RBolSzn}).

For each fixed $r\in \N$, set
$M_r^n(t) :=\#\{$mutations that arrive before time $t$ and affect precisely $r$ individuals$\}$,
and
$M_r^{n;o}(t) := \# \{$open mutations that arrive before time $t$ and affect precisely $r$ individuals$\}$,
and 
\[
M_r^n:= M_r^n(\tau_*^n), \ M_r^{n;o}:= M_r^{n;o}(\tau_*^n).
\]
In the case of Kingman's coalescent, the Ewens
sampling formula \cite{ewens_sam} 
gives the joint law of the above family of random variables as follows:
given $a_i$, $i=1,\ldots, n$ such that $a_i\geq 0$ and $\sum_{i=1}^n i a_i =n$,
\begin{equation}
  \label{ESF}
 P(M_{i}^{n,o}=a_i,\,i=1,\ldots,n) =
  \frac{n!}{2\gamma (2\gamma+1)\cdots(2\gamma+n-1)}\prod_{i=1}^n
  \frac{(2\gamma)^{a_i}}{i^{a_i}a_i!}.
\end{equation} 
For a general exchangeable coalescent one can only hope for 
asymptotic analogues to (\ref{ESF}). 
The first such approximations were given by Berestycki et
al.~\cite{bbs1,bbs2}
for all the Beta-coalescents with locally finite genealogies.

\begin{coro}
\label{Cnewad}
Suppose that $\ell(n) \sim^* n^\beta$ for some $\beta\in (0,1)$,
where $\sim^*$ stands for $\sim$ 
up to a slowly varying multiple.
\begin{equation}
\label{Eas verbal}
\mbox{both $M_r^n$ and $M_r^{n;o}$ are asymptotic to 
$\beta \Gamma(r-\beta) \ell(n)/r!$, in probability, }
\end{equation}
where 
due to assumptions $\Gamma(z)= \int_0^\iy t^{z-1} e^{-t}\,dt$ is well defined at $r-\beta$, $r\in \N$. 
\end{coro}
\begin{proof}[Sketch of the proof]
Note that both
$M_r^n - M_r^n(\tau^n)\in \{0,1\}$, $M_r^{n;o}- M_r^{n;o}(\tau^n)\in
\{0,1\}$,
where these differences are positive for at most one $r \in \{1,\ldots,n\}$. 
One can combine Theorem \ref{Tfull limit} with Theorem 1.2 in
Schweinsberg \cite{schweinsberg_sb} to conclude the claim (see also \cite{bbl2}).
\end{proof}
A stronger version of this result (in the sense of almost sure convergence)
was obtained recently in \cite{bbl2}, applying the main result of \cite{ghp}, 
for a class of $\La$-coalescent with ``$\beta$-regular variation at $0$'',
that comprises the class of the above mentioned Beta-coalescents.
However, $\ell(n)$  might be asymptotic 
to $n^\beta$ even without the $\beta$-regular variation condition.

More generally, if $\ell(n_k)\sim^* (n_k)^\beta$ for some $\beta\in (0,1)$
along a given subsequence $(n_k)_{k\geq 1}$, we obtain the same
asymptotics for $M_r^{n_k}$ and $M_r^{n_k;o}$ as in Corollary \ref{Cnewad}. In some (rather vague) sense,
for each fixed $r$, $M_r^{n_k}$ and $M_r^{n_k;o}$ should both be close to 
$\frac{\log\ell(n_k)}{\log n_k} \Gamma(r- \frac{\log\ell(n_k)}{\log n_k}) \ell(n_k)/r!$

One can generalize (\ref{Eas verbal})
in a different way: 
given any $(t_n)_{n\geq 1}$ such that $\ell_{t_n}(n) \sim^* n^\beta$,
for $\beta\in (0,1)$,
one can combine 
Theorem \ref{Tpartial limit} again with Theorem 2 in \cite{schweinsberg_sb} to 
conclude that 
both $M_r^n(t_n)$ and $M_r^{n;o}(t_n)$ are asymptotic to $\beta
\Gamma(r-\beta) \ell_{t_n}(n)/r!$, in probability. 
Again, this convergence could hold only along subsequences.

It would seem useful for applications to not only know the first order asymptotics but also the error terms, and
there seems to be no good general approach for obtaining these. 
However, for specific $\X$ (that is $\psi$)
one could get concrete  error bounds by redoing the general calculations (in the next section)
with this given $\psi$.

\bigskip
It is not difficult to see that the 
asymptotic results given here depend only on the behavior of measure $\X$ 
``close to $(0,0,\ldots)$''. More precisely, if $\X_1$ and $\X_2$ are two measures on $\Delta$ satisfying
\begin{equation}
\label{E eq close to zero}
\X_1(d\bx)\, \indica{\sum_i x_i \leq \eps} = \X_2(d\bx)\, \indica{\sum_i x_i \leq \eps}
\end{equation}
for some $\eps >0$, then the asymptotic quantities in Theorems \ref{TMinit}--\ref{Tpartial limit} 
corresponding to $\X_1$ and $\X_2$
will be identical.
This will continue to hold even if (\ref{E eq close to zero}) is true only in the $\sim$ sense, in fact, as long as 
$\psi_{\X_1}(q)\sim\psi_{\X_2}(q)$ as $q\to \iy$.

What might be surprising (even disturbing) is that for any probability measure $\X$ on $\Delta$ there exists a 
probability measure $\La$ on $[0,1]$ such that
\[
\psi_\X(q) = \psi_\La (q), \ \forall q\geq 0,
\]
where $\psi_\La(q)$ is defined as $\int_{[0,1]} (e^{-qx} -1 + qx)\, \La(dx)/x^2 $.
The question of whether the measures $\X$ and $\La$ can be connected in terms of a 
stochastic coupling construction remains open.
To verify the above identity, one can check that $f:=\psi_\X'$ is a completely monotone function, meaning that 
$(-1)^n \frac{d^n}{dq^n} f(q) \geq 0$ on $(0,\iy)$.
Then due to Bernstein's theorem $f$ is a Laplace transform of a positive Borel measure on $[0,\iy)$,
and it can be written in the form
\begin{equation}
\label{Epomoc}
\psi_\X'(q) = f(q) = a+ bq + \int_0^\iy (1-e^{-qx}) \,\mu(dx),
\end{equation}
where $\mu$ is a measure on $[0,\iy)$ such that $\int (1\wedge x) \, \mu(dx) <\iy$, and $\mu(\{0\})=0$.
Differentiating (\ref{D:psi xi}) twice gives
$\psi_\X'(0)=0$ and $\psi_\X''(0)=1$, so it must be $a=0$ and 
\[
1=b+\int_0^\iy x \,\mu(dx).
\]
Moreover, $\psi_\X(q)/q^2 \to \X((0,0,\ldots))/2$, thus (\ref{Epomoc}) implies that $b= \X((0,0,\ldots))$.
Then the above identity gives $\X((0,0,\ldots))+\int_0^\iy x \, \mu(dx)=1$, so $\La(dx):=x\,\mu(dx) + \X((0,0,\ldots)) \delta_0(dx)$ is a probability measure on $[0,\iy)$.
By (\ref{Epomoc}) and the fundamental theorem of integral calculus one now has identity
\[
\psi_\X(q) =  b \frac{q^2}{2}+\int_0^\iy \int_0^q (1- e^{-tx}) \, \mu(dx) = b \frac{q^2}{2}+\int_0^\iy \frac{qx+e^{-qx}-1}{x}\,\mu(dx).
\]
The last quantity is identical to $\psi_\La$, provided that the support of $\La$ (that is, of $\mu$)
is the unit interval.
This last claim can be checked by using the inverse Laplace transform formula (see, e.g., \cite{durrett}, Example 5.4)
together with the definition of $\psi_\X$.

\section{The arguments}
\label{S:martingale}
The proof of Proposition \ref{PN} is based on the martingale technique from \cite{me_xi},
that originated in \cite{bbl1} in the $\La$-coalescent setting.
More precisely, 
Proposition 17 in \cite{me_xi} shows existence of $n_0\in \N$ and $C \in (0,\iy)$, such that
\begin{equation}
E[d\log(N^n(s))\,|\,\FF_s]= \left(-\frac{\psi(N^n(s))}{N^n(s)} + h^n(s)\right) ds,
\label{Eha}
\end{equation}
where $(h^n(s), s\ge z)$ is an $\FF$-adapted process satisfying
$\sup_{s\in [z,z \wedge \tau_{n_0}^n]} |h^n(s)|\le C$, uniformly over $n$, and where
\begin{equation}
\label{EvarN}
E[[d\log(N^n(s))]^2\,|\,\FF_s] \indic{\{s \leq \tau_{n_0}^n\}} \leq C\, ds, \mbox{ almost surely.}
\end{equation}
A crucial new observation is that $N^{n;o}$ behaves analogously.
Indeed, for $\bx \in \Delta$, let
$(X_j,\, j \in \N)$ be a family of i.i.d.~(generalized) random variables 
with law $\P_{\bx}$, where
\[
\P_\bx(X_1=i) = x_i, \ i\in \N, \mbox{ and } \P_\bx(X_1=\iy) = 1- \sum_{i=1}^\iy x_i,
\]
Furthermore, for each $b\in \N$, let the family $(Y_\ell^{(b)},\,i\in \N)$ of random variables
be defined by
\begin{equation}
\label{EYs}
Y_\ell^{(b)}:=\sum_{j=1}^b \indica{X_j=\ell}, \ \ell \in \N.
\end{equation}
Then we have, 
on the event $\{N^{n;o}(s)=b\}$ 
(see the proof of \cite{me_xi}, Proposition 17), 
\begin{eqnarray}
\nonumber
E(d \log(N^{n;o}(s))\,|\,\FF_s)
\bck &=& \bck
\int_{\Delta} E_{\bx}\!\!\left[\log\frac{b-\sum_{\ell=1}^\infty (Y_\ell^{(b)} -\indica{Y_\ell^{(b)} >0})}{b} \right]\!\!\frac1{\sumd}\,\X'(d\bx) \,ds\\
\label{EdlogNo}
\bck &+& \bck
(1-\X'(\Delta)) {b \choose 2} \log\frac{b-1}{b}
+ \log\frac{b-1}{b} \cdot \gamma b.
\end{eqnarray}
The first two terms in the drift above are identical to the only terms in the expression for the infinitesimal drift of 
$\log(N^n)$ (at time $s$ on the event $\{N^n(s)=b\}$), leading to (\ref{Eha}).
The third term comes from the additional loss (that is, closure) of open branches 
at constant rate $\gamma$. 
Since $b \log((b-1)/b) = -1 + o(1)$, one obtains
\begin{equation}
E[d\log(N^{n;o}(s))\,|\,\FF_s]= \left(-\frac{\psi(N^{n;o}(s))}{N^{n;o}(s)} + h_o^n(s)\right) ds,
\label{Eha open}
\end{equation}
for $n_0,C$ from (\ref{Eha}) and $C_o=C+\gamma$, and where
$\sup_{s\in [z,z \wedge \tau_{n_0}^{n;o}]} |h_o^n(s)|\le C_o$, uniformly over $n$.
Moreover, since 
$b \log^2((b-1)/b) = o(1)$, the bound from (\ref{EvarN}) holds with $N^{n;o}$ (resp.~$C_o$)
in place of $N^n$ (resp.~$C$).

Recall the map $v^n$ defined in (\ref{Evn}).
The argument of \cite{me_xi}, Section 4.1, Part I already yields (cf.~display (40) in \cite{me_xi})
\begin{equation}
\label{Esuffices n}
P\left(\sup_{t\in[0,s]} \left|
\log{\frac{N^n(t\wedge \tau_{n_0}^n)}{v^n(t \wedge \tau_{n_0}^n)}} \right|
> 2s^{\alpha}
\right) = O(s^{1-2\alpha}),
\end{equation}
but it clearly carries over to imply
\begin{equation}
\label{Esuffices n open}
P\left(\sup_{t\in[0,s]} \left|
\log{\frac{N^{n;o}(t\wedge \tau_{n_0}^{n;o})}{v^n(t \wedge \tau_{n_0}^{n;o})}} \right|
> 2s^{\alpha}
\right) = O(s^{1-2\alpha}).
\end{equation}
Since for $\eps \in (0,1]$ we have
$|\log(x)|\leq \eps$ iff $1-x\in [0,1-e^{-\eps}]$ implying $1-x \in [0,\eps]$, 
or $x-1\in [0,e^{\eps}-1]$ implying $x-1 \in [0,2\eps]$, 
one obtains (\ref{Ebehav N})
by combining the estimates in (\ref{Esuffices n}) and (\ref{Esuffices n open}), and noting that 
$\tau_{n_0}^{n;o}\leq \tau_{n_0}^{n}$, $\forall n, n_0$ in the full genealogy coupling.
Alternatively, the same could be concluded by applying the argument leading to 
(\ref{Esuffices n})--(\ref{Esuffices n open}) in terms of the process
$\log(N^n/N^{n;o})$, that is, by comparing directly $N^n$ and $N^{n;o}$.

\begin{remark}
\label{Rrobust}
It is important to note that if the mutation rate per unit length were not constant but given 
instead as a non-negative $\FF$-adapted stochastic process
$(\gamma_t, t\geq 0)$ such that $\P(\sup_{t\geq 0} 	\gamma_t \leq \gamma)=1$ for some $\gamma<\iy$, then 
(\ref{EdlogNo}) would become
\begin{eqnarray}
E(d \log(N^{n;o}(s))\,|\,\FF_s)
\nonumber
\bck &=& \bck
\int_{\Delta} E_{\bx}\!\!\left[\log\frac{b-\sum_{\ell=1}^\infty (Y_\ell^{(b)} -\indica{Y_\ell^{(b)} >0})}{b} \right]\!\!\frac1{\sumd}\,\X'(d\bx) \,ds\\
\label{EdlogNor}
\bck &+& \bck
(1-\X'(\Delta)) {b \choose 2} \log\frac{b-1}{b}
+ \log\frac{b-1}{b} \cdot \gamma_s \cdot b.
\end{eqnarray}
Since $\log\frac{b-1}{b} \cdot \gamma_s \cdot b \leq \gamma$, the rest of the proof of Proposition \ref{PN} would remain identical. Moreover, it is simple to see that the 
forthcoming arguments would easily carry over to yield appropriate analogues of 
Theorems \ref{TMinit}, \ref{Tfull limit} and \ref{Tpartial limit} in this general setting, 
where 
$\gamma\int N^n(u)\,du$ is being replaced by $\int \gamma_u N^n(u)\,du$
and
$\gamma \cdot \ell(n)$ by $\ell(n;\gamma):=\int_0^1 \E(\gamma_u)v^n(u)\,du$, under additional hypotheses on the process
$\gamma$ that would guarantee $\int_0^1 \gamma_u v^n(u)\,du \sim \ell(n;\gamma)$, as well as the divergence of the sequence $(\ell(n;\gamma))_{n\geq1}$.
For example, for estimate (\ref{Epoidomi}) one would use the stochastic domination with the 
Poisson (mean $\gamma \int_0^t N^{n;c}(u)\,du$) law, instead of the law itself.

As mentioned in the Introduction, Proposition \ref{PN} is just an example of a 
fairly general result that could be stated as follows: 
if the full genealogy dynamics is enriched so that in the new process the
branches are called either ``open'' or ``closed'', and the open branch count process satisfies (\ref{EdlogNor}), 
then (\ref{Ebehav N}) holds.
Two recent specific examples from this class of results are \cite{foucart} (verifying a weaker coming down from infinity criterion for a $\La$-coalescent with migration model) and \cite{parsal} (deriving the speed of coming down from infinity for the Kingman coalescent with constant recombination rate).
The proof of this ``meta-theorem'' is essentially given in the previous paragraph.
\hfill $\Diamond$
\end{remark}

{\em Proof of Theorem \ref{TMinit}.}
Due to the construction of the full marked genealogy coupling, we know that, 
given the path of the processes $N^n$ up to time $t$, 
$M^n(t)$
is a Poisson (mean $\gamma \int_0^t N^n(u)\,du$) random variable, that can be obtained as a sum of 
$M^{n;o}(t)$ and  $M^{n;c}(t)$.
Moreover, if $N^{n;o}$ (that is $N^{n;c}=N^n-N^{n;o}$) is given in addition, 
then 
$M^{n;c}(t)$ is a Poisson (mean $\gamma \int_0^t N^{n;c}(u)\,du$) random variable, conditionally 
independent of $M^
{n;o}(t)$.
The last observation is true if the fixed time $t$ above is replaced by the random time $t \wedge \tau_{n_0}^{n;o}$, measurable 
with respect to $\sigma\{N^n(u),N^{n,o}(u),\, u\leq t\}$.
Note that $M^{n;o}(t)$ cannot have (conditional) Poisson distribution, since $\P(M^{n;o}(t)\leq n)=1$.
Note in addition that $M^n(t \wedge \tau_{n_0}^n)$
is a Poisson (mean $\gamma \int_0^{t \wedge \tau_{n_0}^n} N^n(u)\,du$) random variable, given 
$\FF_{\tau_{n_o}^n}$.
We cannot say the same if $\tau_{n_0}^n$ is replaced here by $\tau_{n_0}^{n;o}$, since
knowing $N^n(\tau_{n_0}^{n;o})>n_0$ (that is, at least one branch is closed at time $\tau_{n_0}^{n;o}$)
excludes the event $\{M^n(\tau_{n_0}^{n;o})=0\}$ that no mutation arrived prior to time $\tau_{n_0}^{n;o}$. Due to this, some additional technical steps are needed below (notably, in arguing (\ref{Etechnia})--(\ref{Etechnib})).

Due to Proposition \ref{PN}, on the event $A_t= \{N^{n;c}(u)\leq 8t^\alpha N^n(u),\forall u\in [0,t \wedge \tau_{n_0}^{n;o}]\}= 
\{N^{n;o}(u)\geq (1-8t^\alpha) N^n(u),\forall u\in [0,t\wedge \tau_{n_0}^{n;o}]\}$ of probability greater than
$1-O(t^{1-2\alpha})$ we have
\[
\int_0^{t\wedge \tau_{n_0}^{n;o}} N^{n;c}(u)\,du \leq 8t^\alpha \cdot \int_0^{t\wedge \tau_{n_0}^{n;o}} N^n(u)\,du.
\]
On the complement of $A_t$ we have $\int_0^{t\wedge \tau_{n_0}^{n;o}}
N^{n;c}(u)\,du \leq \int_0^{t\wedge \tau_{n_0}^{n;o}} N^n(u)\,du$, since
$N^{n;c} \leq N^n$ as a consequence of the full marked genealogy coupling.
Hence we compute
\begin{eqnarray}
\nonumber
E\left[ \frac{M^{n;c}(t\wedge \tau_{n_0}^{n;o})}{\int_0^{t\wedge \tau_{n_0}^{n;o}} N^n(u)\,du}\right] \bck&=&\bck 
E \left[E\left(\left. \frac{M^{n;c}(t \wedge \tau_{n_0}^{n;o})}{\int_0^{t\wedge \tau_{n_0}^{n;o}} N^n(u)\,du } \,\right|\, ((N^n(u),N^{n;o}(u)),\,u\in [0,t \wedge \tau_{n_0}^{n;o}]) \right) \right]\\
\label{Epoidomi}
\bck&=&\bck E\left[ \frac{\gamma \int_0^{t\wedge \tau_{n_0}^{n;o}} N^{n;c}(u)\,du}{\int_0^{t\wedge \tau_{n_0}^{n;o}} N^n(u)\,du} \right]\\
\nonumber
\bck&\leq&\bck \gamma \left( 8t^\alpha \P(A_t) + \P(A_t^c)\right),
\end{eqnarray}
and the estimate (\ref{EMinit one}) readily follows.

In order to obtain (\ref{EMinit two}), we first note that
\begin{eqnarray*}
\Bigg\{ \!M^n(t \wedge \tau_{n_0}^{n;o})  < \frac{\gamma}{4} \!\!\! & &\!\!\!\!\!\!\left. \int_0^{t\wedge \tau_{n_0}^{n;o}} N^n(u)\,du \right\} \subset 
\left\{.\!M^n(t \wedge \tau_{n_0}^n) < \frac{\gamma}{2} \int_0^{t\wedge \tau_{n_0}^n}\!\! N^n(u)\,du \right\} \bigcup \\
& & \left\{ M^n(t \wedge \tau_{n_0}^n) - M^n(t \wedge \tau_{n_0}^{n;o}) > \frac{\gamma}{4} \int_0^{t\wedge \tau_{n_0}^{n;o}}\!\! N^n(u)\,du  \right\}.
\end{eqnarray*}
The estimate (\ref{EMinit two}) will now simply follow
from (\ref{EMinit one}), the Markov inequality, and
\begin{eqnarray}
\label{Etechnia}
& &\limsup_n \P\left( M^n(t \wedge \tau_{n_0}^n) < \frac{\gamma}{2} \int_0^{t\wedge \tau_{n_0}^n}\!\! N^n(u)\,du \right)=0,\\
\label{Etechnib}
& &\limsup_n \P\left( M^n(t \wedge \tau_{n_0}^n) - M^n(t \wedge \tau_{n_0}^{n;o}) > \frac{\gamma}{4} \int_0^{t\wedge \tau_{n_0}^{n;o}}\!\! N^n(u)\,du \right)=0.
\end{eqnarray}
To obtain (\ref{Etechnia}), we first note that
\[
\P\left(\!\left.\!M^n(t \wedge \tau_{n_0}^n) < \frac{\gamma}{2} \int_0^{t\wedge \tau_{n_0}^n}\!\! N^n(u)\,du \,\right| (N^n(u),\,u \leq t \wedge \tau_{n_0}^n) \!\right)
\leq E(e^{-c \int_0^{t \wedge \tau_{n_0}^n} N^n(u)\,du }),
\]
for some $c>0$, uniformly over $n$ (and $\gamma$), and also
that $(\int_0^{t\wedge \tau_{n_0}^n} N^n(u)\,du)_{n\geq 1}$ diverges, almost surely, as $n\to \iy$.
Indeed, this sequence of random variables
is monotone increasing in $n$, so it must be
converging, almost surely, to a possibly generalized (i.e., taking value $\iy$) random variable.
Moreover, the divergence with probability $1$ is clear if the full genealogy is not locally finite.
Otherwise, 
note that $(\int_0^t N^n(u)\,du)_{n\geq 1}$ diverges while $(\int_{t\wedge \tau_{n_0}^n}^t N^n(u)\,du)_{n\geq 1}$
converges, almost surely, as $n\to \iy$.
The latter statement is clear due to (\ref{Etau fin}). 
The former follows from the fact
that $\lim_n v^n(t)=v(t)<\iy$, the inequality (\ref{Epsi ineq}), the asymptotics
(\ref{Esuffices n}), and 
\begin{equation}
\label{Eit diver}
\int_0^t v^n(u)\,du = \int_{v^n(t)}^n \frac{q}{\psi(q)}\,dq \geq \int_{v^n(t)}^n \frac{2}{q}\,dq \approx 
\int_{v(t)}^\iy \frac{2}{q}\,dq.
\end{equation}
The first identity above, as noted in \cite{bbl2}, is due to the change of variables $q= v^n(u)$, $dq= - \psi(v^n(u)) \,du$.

To obtain (\ref{Etechnib}), note that
$M^n(t \wedge \tau_{n_0}^n) - M^n(t \wedge \tau_{n_0}^{n;o})$ is a Poisson (mean $\gamma \int_{t\wedge \tau_{n_0}^{n;o}}^{t\wedge \tau_{n_0}^n} N^n(u)\,du$) random variable given $\FF_{\tau_{n_0}^n}$. 
Moreover, in the setting of locally finite full genealogy, due to (\ref{Etau o fin}), the sequence $(\int_{ t \wedge \tau_{n_0}^{n;o}}^{t\wedge \tau_{n_0}^n} N^n(u)\,du)_{n\geq 1}$ is convergent (tight), while $(\int_0^{t\wedge \tau_{n_0}^{n;o}}\!\! N^n(u)\,du)_{n\geq 1}$ is always divergent, implying (\ref{Etechnib}).
Otherwise, both $\tau_{n_0}^{n;o}$ and $\tau_{n_0}^n$ diverge to $\iy$ as $n\to \iy$, almost surely, so 
$M^n(t \wedge \tau_{n_0}^n) - M^n(t \wedge \tau_{n_0}^{n;o})=0$, except on an event of negligible probability, again implying (\ref{Etechnib}).

{\em Proof of Theorem \ref{Tfull limit}.}
The same change of variables as in (\ref{Eit diver}) gives
\[
\int_0^{(v^n)^{-1} (1)} v^n(u) \,du = \int_1^n \frac{q}{\psi(q)}\,dq = \ell(n),
\]
where 
\[
(v^n)^{-1} (1) = \int_{1}^n \frac{dq}{\psi(q)}.
\]
The hypotheses of locally finite full genealogy and regularity imply 
(see \cite{schweinsberg_xi}, Proposition 33 or \cite{me_xi}, Theorem 1)
that  $\int_a^\iy dq/\psi(q) < \infty$ 
(for any $a>0$), and therefore that $(v^n)^{-1} (1)$ increases towards a finite limit $1^*$.
As a consequence, $\ell(n)$ diverges as $n\to \iy$ (see (\ref{Eit diver})).
\begin{remark}
Note that $\ell(n) \sim \int_0^a v^n(u)\,du$ for any fixed $a>0$.
\hfill $\Diamond$
\end{remark}
It suffices to prove the convergence in probability for 
$M^n(\tau^n)/ (\gamma \ell(n))$
and 
$M^{n;o}(\tau^n)/ (\gamma \ell(n))$.
We discuss the convergence of the former sequence in some detail below, and give at present the reasoning 
for the latter convergence assuming the former: 
since $M^n(\tau^n)= M^{n;o}(\tau^n) + M^{n;c}(\tau^n)$, it suffices to show that 
$M^{n;c}(\tau^n) = o(\ell(n))$, in probability.
This will follow by (\ref{EMinit two}), provided $M^{n;c}(\tau^n)- M^{n;c}(\tau_{n_0}^{n;o} \wedge t) = o(\ell(n))$
for some fixed $t$.
However, $M^{n;c}(\tau^n)- M^{n;c}(\tau_{n_0}^{n;o} \wedge t) \leq
M^n(\tau^n)- M^n(\tau_{n_0}^{n;o} \wedge t)$ in the full (marked) genealogy 
coupling, and $M^n(\tau^n)- M^n(\tau_{n_0}^{n;o} \wedge t)$ is (as in the proof of Proposition \ref{PN}) a conditional Poisson random variable with 
mean
$\gamma \cdot \int_{\tau_{n_0}^{n;o} \wedge t}^{\tau^n} N^n(u)\,du$.
So the above claim follows due to convergence (tightness)
of $\int_{\tau_{n_0}^{n;o} \wedge t}^{\tau^n} N^n(u)\,du$ and divergence of $\ell(n)$.

The rest of the argument is analogous to that for Theorem 5 in \cite{bbl1}.
The bulk of it is showing that $\int_0^{\tau^n} N^n(u)\, du \sim \ell(n)$, in probability, as $n\to \iy$. 
Since, given $N^n$, $M^n(\tau^n)$ is a Poisson random variable with mean 
$\int_0^{\tau^n} N^n(u)\,du$, and since $\ell(n)$ diverges with $n$, one concludes (as a special case of LLN)
that $M^n(\tau^n)\sim \gamma \ell(n)$, in probability.
It therefore suffices to show that for any subsequence $n_k$ there exists a further subsequence $n_{k_j}$ such that
$\int_0^{\tau^{n_{k_j}}} N^{n_{k_j}}(u)\,du/\ell(n_{k_j}) \to 1$, almost surely.

To simplify the notation, we rename the subsequence $(n_k)_{k\geq 1}$ as $(k)_{k\geq 1}$.
Recall (\ref{Esuffices n}), and for a fixed $\alpha \in (0,1/2)$ choose a decreasing sequence 
$(s_k)_{k\geq 1}$ of positive numbers, such that
\[
\sum_k s_k^{1-2\alpha} < \iy.
\]
Then we can conclude from (\ref{Etau fin}), (\ref{Esuffices n}) and the Borel-Cantelli lemma that
\begin{equation}
\label{Eunif conv as}
\lim_{k\to \iy} \sup_{t\in[0,s_k]} \left|
\frac{N^{k}(t)}{v^k(t)} -1 \right|
=0, \mbox{ almost surely.}
\end{equation}
Due to the assumption of locally finite full genealogy, 
we can now choose a subsequence $k_j$ such that $\int_0^{s_j} v^{k_j}(u)\,du$ diverges as $j\to \iy$, and that also 
both 
\begin{equation}
\label{Esmalltointertime}
\int_{s_j}^{1^*} v^{k_j}(u)\,du \,\,/ \int_0^{s_j} v^{k_j}(u)\,du \mbox{ and }
\int_{s_j}^{\tau^{k_j}} N^{k_j}(u)\,du \,\, / \int_{0}^{s_j} N^{k_j}(u)\,du
\end{equation}
tend to $0$ as $j\to \iy$ (for the second sequence, the limit is taken almost surely).
Joint with
(\ref{Eunif conv as}), this ensures that
\[
\lim_{j\to \iy} \frac{\int_0^{\tau^{k_j}} N^{k_j}(u)\, du}{\int_0^{1^*} v^{k_j}(u)\, du} =
\lim_{j\to \iy} \frac{\int_0^{s_j} N^{k_j}(u)\, du}{\int_0^{s_j} v^{k_j}(u)\, du} = 1,
\mbox{ almost surely.}
\]

\begin{remark}
\label{RBolSzn}
If $\X(d\bx)= \indica{\bx=(x,0,0,\ldots)}$ for $x\in (0,1)$, or equivalently, in case of the Bolthausen-Sznitman coalescent
we have
\[
\psi(q) = q\log{q} + O(q), \mbox{ as } q \to \infty,
\]
with $O(q)\geq 0$, for all $q>0$.
Therefore
\[
v^n(t) \in [n^{e^{-t(1+o(1))}},n^{e^{-t}}] \mbox{ as well as } \ell(n) \sim \frac{n}{\log{n}}, \mbox{ as } n \to \infty.
\]
We conclude that $v^n(t)$ is of order $1$ at times of order $\log\log(n)$ and in turn that
\begin{equation}
\label{ERoneimpo}
\int_1^{\log\log n }v^n(t)\,dt = o(\ell(n)), \mbox{ as } n \to \iy.
\end{equation}
In order to show that $\int_0^{\tau^n} N^n(u)\,du \sim \ell(n)$, it therefore suffices to start as in the paragraph which comprises (\ref{Eunif conv as})--(\ref{Esmalltointertime}), ensuring the following analogue of (\ref{Esmalltointertime})
\[
\int_{s_j}^{1} v^{k_j}(u)\,du \,\,/ \int_0^{s_j} v^{k_j}(u)\,du \to 0 \mbox{ and }
\int_{s_j}^{1} N^{k_j}(u)\,du \,\, / \int_{0}^{s_j} N^{k_j}(u)\,du \to 0,
\]
and to verify in addition that (possibly along a further subsequence)
\begin{equation}
\label{ERtwoimpo}
\int_{1\wedge \tau^{k_j}}^{\tau^{k_j}}N^{k_j}(t)\,dt = o(\ell(k_j)), \mbox{ as } j \to \iy.
\end{equation}
This can all be done due to (\ref{Esuffices n}), the fact that $v^n(t)$ is bounded by a power of $n$ smaller than $1$ for each fixed $t$, asymptotically in $n$, and finally the estimate 
\[
E\tau^n = O(\log\log{n}),
\]
which can be obtained via optional stopping of the martingale $\bar{M}_t:= \int_{N^n(t)}^n \frac{dq}{\bar{\psi}(q)} -t$, $t\geq 0$, where $\bar{\psi}(q):= \int_{[0,1]} ((1-x)^q - 1 + qx)/x^2=\psi(q) + O(1)$ 
(see \cite{mahabi} for further use of $\bar{M}$ and its generalizations).
\hfill $\Diamond$
\end{remark}

{\em Proof of Theorem \ref{Tpartial limit}.}
Due to the assumption that $\ell_{t_n}(n)$ diverges, one can argue as in the proof of Theorem \ref{Tfull limit}
that it suffices to show that
\begin{equation}
\label{Esuffices show}
\frac{M^n(t_n)}{\gamma\cdot \ell_{t_n}(n)}\to 1, \mbox{ in probability}.
\end{equation}
The argument for (\ref{Esuffices show}) is analogous to the last one.
In fact, with the same choice of the sequence $(s_k)_{k\geq 1}$ as above, one now chooses a subsequence $k_j$ so that
$\int_0^{s_j \wedge t_j} v^{k_j}(u)\,du$ diverges as $j\to \iy$, and in addition 
both
\[
\int_{s_j}^{t_j} v^{k_j}(u)\,du \, \indica{t_j>s_j} \,\,/ \int_0^{s_j} v^{k_j}(u)\,du \mbox{ and }
\int_{sj}^{t_j} N^{k_j}(u)\,du \, \indica{t_j>s_j}  \,\, / \int_{0}^{s_j} N^{k_j}(u)\,du
\]
tend to $0$ as $j\to \iy$.
Joint with
(\ref{Eunif conv as}) this ensures that
\[
\lim_{j\to \iy} \frac{\int_0^{t_j} N^{k_j}(u)\, du}{\int_0^{t_j} v^{k_j}(u)\, du} =
\lim_{j\to \iy} 
 \frac{\int_0^{t_j\wedge s_j} N^{k_j}(u)\, du}{\int_0^{t_j \wedge s_j} v^{k_j}(u)\, du} = 1,
\]
almost surely.
Finally, (\ref{Esuffices show}) follows by another application of LLN. 

\begin{remark}
\label{R:ending}
The arguments of \cite{bbl1}, Section 4.3 apply in the regular setting. 
Indeed, instead of
\cite{bbl1}, Lemma 20 one now has the following statement:
 There exists $n_0\in \N$ and $K_0<\infty$ such
that for all $b\geq n_0$, $\bx \in \Delta \cap \{\bx:\sum_{i=1}^\iy x_i \leq 1/4\}$, $c>0$, and 
$Y_\ell^{(b)}$ given by (\ref{EYs}) we have
\begin{equation}
\label{Eboundica}
E_\bx[ \exp\{c[\log(b- \sum_{\ell=1}^\iy (Y_\ell^{(b)} - \indica{Y_\ell^{(b)} >0}))  - \log{b}]^2\} -1] \leq e^{9c/4}
K_0 [(\sum_i x_i^2) + (\sum_i x_i)^2].
\end{equation}
Since $\sum_i x_i^2 \leq (\sum_i x_i)^2$, the RHS above could be simply bounded by $2K_0 (\sum_i x_i)^2$, however, in doing so one might lose some information of the impact of regularity (or irregularity).
Due to (\ref{Eboundica}),
in the definition of the process $E^{(c)}$ and the related calculations 
leading to (33)--(34) in \cite{bbl1},
the constant $K_0$ should be replaced by 
\[
\bar{K_0} = K_0\left(1+ \int (\sum_{i=1}^\iy x_i)^2\frac{\Xi(d\bx)}{\sumd}\right).
\]
In particular, one can conclude that
for any $s>0$ and $d\geq 1$ 
\begin{equation}
\label{E all moments}
\sup_{n \geq 1} \E \left(\sup_{t\in [0,s]} \left|\frac{N^{n}(t)}{v^n(t)}\right|^d\right) = D(d,s) <\iy, 
\end{equation} 
and moreover that $\lim_{s\to 0} D(d,s)=0$.
As indicated earlier, we then have $N^n(t)/v^n(t) \to 1$, 
as $t\to 0$ in $L^d$, for each $d\geq 1$.

As a consequence of the previous observations, and arguments very similar to those for Theorem \ref{TMinit},
it is not difficult to check that, for each fixed $t>0$,
\begin{equation}
\label{EL1limit M}
\frac{E(M^n(t))}{\gamma \cdot \ell_t(n)} = \frac{E(\int_0^t N^n(u)\,du)}{\int_0^t v^n(u)\,du} \to 1,
\end{equation}
as well as 
\[
\frac{E(M^{n;c}(t))}{\gamma \cdot \ell_t(n)} = \frac{E(\int_0^t N^{n;c}(u)\,du)}{\int_0^t v^n(u)\,du} \to 0,
\]
implying (\ref{EL1limit M}) with $M^{n;o}$ in place of $M^n$.
Under the assumption of locally finite genealogies,
$\ell_1(n) \sim \ell(n)$ and 
\[
E[M^n(\tau_1^n)-M^n(1)] = E\left[\gamma \!\int_1^{\tau_1^n}\! N^n(u)\,du\right]\leq  \gamma E[N^n(1)\cdot (\tau_1^n-1) ]\leq C,
\ \forall n\geq 1,
\]
where the final uniform estimate is due to $\sup_n E(N^n(1))<\iy$, and
$\sup_n E(\tau_1^n-1|N^n(1))\leq E(\sup_n \tau_1^n) <\iy$.
Hence, for a locally finite $\X$-genealogy, both
$M^n(\tau_1^n))/(\gamma \ell(n))$ and $M^{n;o}(\tau_1^n)/( \gamma \ell(n))$ converge to $1$ in the mean.
\hfill $\Diamond$
\end{remark}

{\bf Acknowledgments.} The author wishes to thank Matthias Birkner for a pointer to Bernstein's theorem, and to Julien and Nathana{\"e}l Berestycki for rewarding discussions. She is also grateful to the anonymous referee for several useful comments and pointers to the literature.

\end{document}